\documentclass[11pt]{amsart}

\usepackage{amsthm}
\usepackage{amsmath}
\usepackage{amssymb}
\usepackage[mathscr]{euscript}
\usepackage{mathrsfs}
\usepackage{hyperref}

\usepackage[all]{xy}

\usepackage{latexsym}
\usepackage{framed}

\usepackage{mathabx}

\usepackage{bm}

\title[Matroidal Schur Algebras]{Matroidal Schur Algebras}

\author{Tom Braden}
\author{Carl Mautner}

\address{Tom Braden \\
Dept.\ of Mathematics and Statistics\\
         University of Massachusetts, Amherst}
\email{braden@math.umass.edu}

\address{Carl Mautner \\
Department of Mathematics, University of California, Riverside}
\email{mautner@math.ucr.edu}

\DeclareFontFamily{U}{mathx}{\hyphenchar\font45}
\DeclareFontShape{U}{mathx}{m}{n}{
      <5> <6> <7> <8> <9> <10>
      <10.95> <12> <14.4> <17.28> <20.74> <24.88>
      mathx10
      }{}
\DeclareSymbolFont{mathx}{U}{mathx}{m}{n}
\DeclareFontSubstitution{U}{mathx}{m}{n}
\DeclareMathAccent{\widecheck}{\mathord}{mathx}{"71} 
 
\newtheorem{theorem}{Theorem}[section]
\newtheorem{lemma}[theorem]{Lemma}
\newtheorem{proposition}[theorem]{Proposition}
\newtheorem{corollary}[theorem]{Corollary}

{
\theoremstyle{definition}
\newtheorem{definition}[theorem]{Definition}

\newtheorem{question}[theorem]{Question}

\newtheorem{remark}[theorem]{Remark}
}

\newcommand{\excise}[1]{}

\newcommand{\rk}{\operatorname{rk}}

\renewcommand{\Im}{\operatorname{Im}}

\newcommand{\id}{\operatorname{id}}

\newcommand{\Ker}{\operatorname{Ker}}
\renewcommand{\dim}{\operatorname{dim}}

\newcommand{\Hom}{\operatorname{Hom}}
\newcommand{\End}{\operatorname{End}}

\newcommand{\Z}{{\mathbb{Z}}}

\newcommand{\D}{\mathbb D}

\newcommand{\F}{\mathcal{F}}

\newcommand{\cal}{\mathcal}

\newcommand{\cB}{{\cal B}}

\newcommand{\cF}{{\cal F}}

\newcommand{\cI}{{\cal I}}

\newcommand{\ol}{\overline}

\newcommand{\la}{\langle}
\newcommand{\ra}{\rangle}

\newcommand{\wt}{\widetilde}

\renewcommand{\setminus}{\smallsetminus}

\newcommand{\kk}{{ k}}

\renewcommand{\emptyset}{\varnothing}

\DeclareMathOperator{\bo}{bo}
\DeclareMathOperator{\ci}{ci}
\DeclareMathOperator{\ch}{ch}

\newcommand{\Bas}{\mathop{\mathrm{Bas}}\nolimits}

\newcommand{\cU}{{\mathcal U}}
\newcommand{\cUc}{{\widecheck{\mathcal U}}}
\newcommand{\Uc}{\widecheck{U}}
\newcommand{\Rc}{\widecheck{R}}

\newcommand{\pic}{\widecheck{\pi}}

\newcommand{\ldiv}{\dashv}
\newcommand{\rdiv}{\vdash}
\newcommand{\uc}{{\check{u}}}

\newcommand{\lcon}{\dashV}
\newcommand{\rcon}{\Vdash}

\begin{document}

\maketitle

\section{Introduction}

Fix a principal ideal domain $k$.  In this article we associate to a (weighted) matroid $M$ a quasi-hereditary algebra $R(M)$ defined over $k$ such that matroid duality corresponds to Ringel duality of quasi-hereditary algebras. The representation theory of these algebras is related to work of Schechtman-Varchenko~\cite{SV} and Brylawski-Varchenko~\cite{BV}.  In characteristic zero, our algebras are also closely related to work of Kook-Reiner-Stanton~\cite{KRSlap} and Denham~\cite{Denham}.

While the contents of this paper are purely of an algebraic and combinatorial nature, the algebras described here are a generalization of algebras we discovered in a study~\cite{BMHyperRingel} of the geometry of hypertoric varieties.  In the remainder of this introduction, we first briefly explain our motivation, which comes from the theory of symplectic duality~\cite{BLPWgco} and a geometric description of the Schur algebra~\cite{Mautner,AM}.  We then give a summary of our results and describe the structure of the paper.

\subsection{Background and motivation} 
The Schur algebra $S_{k}(n,n)$ is a quasi-hereditary algebra that plays an important role in the modular representation theory of the general linear and symmetric groups.  There is a natural duality on quasi-hereditary algebras, called Ringel duality, and the Schur algebra $S_k(n,n)$ is its own Ringel dual.

In~\cite{Mautner}, the second author gives a geometric interpretation of the representation theory of the Schur algebra $S_{k}(n,n)$ in terms of equivariant perverse sheaves with coefficients in $k$ on the nilpotent cone $\mathcal{N} \subset \mathfrak{gl}_n$.  Using this description, Achar and the second author~\cite{AM} give a geometric proof of the self-Ringel duality of the Schur algebras $S_k(n,n)$.

The nilpotent cone $\mathcal{N}$ is an important example of a \textit{symplectic singularity}.  An early observation in geometric representation theory was that one can study the universal enveloping algebra $U(\mathfrak{gl}_n)$ in the context of quantization of $\mathcal{N}$.  Recently, much work in geometric representation theory has been focused on studying noncommutative algebras arising from other symplectic singularities via quantization.  In \cite{BLPWgco}, the first author, Licata, Proudfoot and Webster conjecture the existence of a duality for symplectic singularities, relating the quantizations of symplectic dual singularities.  This idea is worked out in detail in the case of hypertoric varities in~\cite{GDKD}.

Motivated by these ideas, we began to study a category of perverse sheaves on affine hypertoric varieties.  In view of the connection between the Schur algebra and perverse sheaves on $\mathcal{N}$, we consider this category to be a \textit{hypertoric} analogue of representations of the Schur algebra.  In~\cite{BMHyperRingel}, we prove that it is highest weight and that its Ringel dual is the corresponding category associated to the Gale dual affine hypertoric variety.

The goal of the current paper is to define and study a class of algebras associated to an arbitrary matroid, generalizing the setting of~\cite{BMHyperRingel}.

Although our original motivation was geometric, the current paper can be read independently of~\cite{BMHyperRingel} and in particular requires no knowledge of geometry, perverse sheaves or hypertoric varieties.

\subsection{Summary and outline}
For any matroid $M$, we define a pair of $k$-algebras $R_k(M)$ and $\Rc_k(M)$, which we dub \textit{matroidal Schur algebras}.  Our main results are:

\begin{theorem}
The algebras $R_k(M)$ and $\Rc_k(M)$ are Ringel dual quasi-hereditary algebras.  There is a natural isomorphism $\Rc_k(M) \cong R_k(M^*)$, where $M^*$ denotes the dual matroid.
\end{theorem}

Section~\ref{sec:Ringel data} contains a general recipe for constructing a Ringel dual pair of quasi-hereditary algebras from certain linear algebra data.  In Section~\ref{sec:matroid}, we show matroids provide an example of such data and thus define the Ringel-dual quasi-hereditary algebras $R_k(M)$ and $\Rc_k(M)$.  The second part of the theorem holds because this construction
is easily seen to be invariant (up to some inconsequential signs) under interchanging the roles of $R$ and $\Rc$ and $M$ and $M^*$.

In~\cite{BMHyperRingel}, we show that the category of perverse sheaves with coefficients in $k$ on a hypertoric variety discussed above is equivalent to the representations of $R_k(M)$ for a corresponding matroid.  In analogy with the result on nilpotent cones, we believe this justifies the name `matroidal Schur algebra.'

Like the original Schur algebra, matroidal Schur algebras are semisimple in all but finitely many charactersitics.  More precisely, we show:

\begin{theorem}
Assume $k$ is a field of characteristic $p$.  The matroidal Schur algebra $R_k(M)$ is semisimple if and only if for any coloop-free flats $K \subset F$ such that $M(F)/K$ is connected, $p$ does not divide the number of elements in $F \setminus K$.
\end{theorem}

Our proof, which appears in Section~\ref{sec:characters}, uses a computation from work of Brylawski-Varchenko~\cite{BV}.  The connection to~\cite{BV} and the closely related paper~\cite{SV} of Schechtman-Varchenko is explained in Section~\ref{sec:BSV form}.

In Section~\ref{sec:characters} we also explain how to compute characters and composition series multiplicities in the representation theory of $R_k(M)$ using the Brylawski-Schechtman-Varchenko result, and give some examples.    The algebras can have non-trivial behaviour, even in relatively small examples.

Our construction has connections to a formula of Kook-Reiner-Stanton~\cite{KRSlap} and work of Denham~\cite{Denham}.  We conclude in Section \ref{sec:KRS formula} with a discussion of this relationship and some open questions.

\subsection{Acknowledgements}
We are grateful to Ben Webster and Geordie Williamson for pointing out to us the connection to the papers~\cite{SV,BV}, which they had found while doing related computations.  We would also like to thank the MPIM in Bonn for excellent working conditions.  
The material in this article is also partly based upon work supported by the NSF
under Grant No. 0932078000, while the second author was in residence at MSRI in Berkeley, California, during the Fall 2014 semester.

\section{Combinatorial setup}\label{sec:Ringel data}

In this section, we give a general framework which allows us to construct Ringel dual pairs of quasi-hereditary algebras.
Fix a principal ideal domain $\kk$.  Throughout this document, we will use $\otimes$ to denote the tensor product $\otimes_\kk$ of $k$-modules.

Our construction takes as input the following data:

\begin{itemize}
\item A finite poset $(\cF, \le)$, 
\item A $\kk$-algebra $\cB$, finitely generated and free as a $k$-module, with multiplication $\ast$ which is graded by the poset, meaning that 
   \begin{enumerate}
   \item $\cB = \bigoplus_{x, y} \cB^x_y$ where $x$ and $y$ are elements of $\cF$,
   \item we have $\cB^x_y \ast \cB^y_z \subset \cB^x_z$ for all $x,y,z \in \cF$, and 
   \item $\cB^x_y \ast \cB^w_z =0$ if $y\neq w$.
   \end{enumerate}
\item Two graded, saturated submodules $\cU = \bigoplus \cU^x_y$ and $\cUc = \bigoplus \cUc^x_y$ of $\cB$, and
\item A symmetric perfect pairing $\langle ,\rangle$ on $\cB$, so that the submodules $\cB^x_y \subset \cB$ are mutually orthogonal.
\end{itemize}

The identity element $1 \in \cB$ decomposes as $1 = \sum_{x\in \cF} 1_x$, where
$1_x \in \cB^x_x$.

We let $\ldiv$, $\rdiv$ denote the adjoint operations to left and right multiplication, respectively.  In other words,
\[\langle a \ldiv c, b\rangle = \langle c, a\ast b\rangle = \langle c \rdiv b , a\rangle\]
for any $a \in \cB^x_y$, $b\in \cB^y_z$, $c \in \cB^x_z$.  We will refer to $\ldiv$ as 	``contracting on the left" and $\rdiv$ as ``contracting on the right".

We require that the data $(\cF, \cB, \cU,\cUc, \langle\rangle)$ satisfy the following axioms.
\begin{enumerate}
\item[A1.] (Triangularity)  $\cB^x_y = 0$ unless $y \le x$.  Furthermore, we have 
$\cB^x_x = \cU^x_x = \cUc^x_x \cong \kk \cdot 1_x$ for all $x$, and   
 $\la 1_x, 1_x \ra = 1$ for all $x$.  
 In particular, this means that 
\[
b_2 \ldiv b_1 = \la b_2, b_1\ra 1_y
\]
 for any $b_1, b_2 \in \cB^x_y$.

\item[A2.] Define $\cU^+ = \bigoplus_{y < x} \cU^x_y$ and $\cUc^+ = \bigoplus_{y < x} \cUc^x_y$.  Then we have
\[\cU^\bot = \cB\ast\cUc^+\;\; \mbox{and}\;\; \cUc^\bot = \cU^+\ast \cB.\]
\item[A3.] (Associativity) For any $u \in \cU$, $b\in \cB$ and $\uc \in \cUc$, we have equalities: 
\begin{align*}
(u \ldiv b)\ast \uc & = u \ldiv (b\ast \uc)\\
u\ast (b \rdiv \uc) & = (u\ast b) \rdiv \uc
\end{align*}
\end{enumerate}

\begin{remark} By adjointness the two equations in A3 are equivalent.  Also, by adjointness and associativity of multiplication on $\cB$, we have
\[(a \ldiv b) \rdiv c  = a \ldiv (b \rdiv c)\]
on all of $\cB \otimes \cB \otimes \cB$, not just on $\cU \otimes \cB \otimes \cUc$.
In the example we have in mind, the equations of A3 will \emph{not} hold on all of $\cB \otimes \cB \otimes \cB$, however.

\end{remark}

We observe one simple consequence of these axioms before continuing with our construction.

\begin{lemma}\label{lemma-A4}
 Both $\cU$ and $\cUc$ are subrings of $\cB$. 
Furthermore, we have containments $\cB \ldiv \cU \subset \cU$ and $\cUc \rdiv \cB \subset \cUc$. 
\end{lemma}
\begin{proof} Take 
$u_1, u_2 \in \cU$. 
Axiom A2 says that to show $u_1 \ast u_2\in \cU$, it is enough to show that \[\langle u_1 \ast u_2 , c\ast \check{u}\rangle = 0\] for any 
$c \in \cB$, $\check{u} \in \cUc^+$.
But we have 
\[\langle u_1\ast u_2 , c \ast \check{u}\rangle = \langle u_2, u_1 \ldiv (c \ast  \check{u})\rangle = \langle u_2, (u_1\ldiv c) \ast \check{u}\rangle = 0\]
by axioms A2 and A3.  The other statements follow similarly.  
\end{proof}

\begin{remark}
The second part of this lemma, which we will not use later, is the only place where we multiply or contract on the left by an element which is not in $\cU$, or on the right by an element which is not in $\cUc$.  In fact, the entire construction would make sense using the weaker structure of partial multiplication maps $\cU \otimes \cB \to \cB$ and $\cB \otimes \cUc \to \cUc$, with appropriate modifications of the axioms.
\end{remark}

We call a tuple $(\cF, \cB, \cU,\cUc, \langle\rangle)$ satisfying A1-A3 a \emph{Ringel datum}.  In this section we show how to use this to construct a Ringel dual pair of quasihereditary $\kk$-algebras $R, \Rc$ such that $\cB$ is an $R$-$\Rc$-bimodule which is simultaneously a tilting generator for $R$ and for $\Rc$.

For each $y \le x$, let $U^{xy} \subset \End_k(\cB)$ be the set of operators $b \mapsto u \ast b$ for $u \in \cU^x_y$, and let $U^{yx} \subset \End_k(\cB)$ be the set of operators $b \mapsto u \ldiv b$.  Since $u \ast 1_x = u$, the natural map $\cU^x_y \to U^{xy}$ is an isomorphism, and by adjunction $U^{yx} \cong \cU^x_y$ as well.

Similarly, 
let $\Uc_{xy}\subset \End_k(\cB)^{\mathrm{opp}}$ be the set of operators $b \mapsto b\ast \check{u}$ and let
$\Uc_{yx}$ be the set of all $b \mapsto b \rdiv \check{u}$, where $\uc$ runs over all elements of $\cUc^x_y$.  As before, both of these spaces are isomorphic to $\cUc^x_y$.

\begin{definition}
Let $R$ (resp. $\Rc$) to be the subalgebra of $\End_k(\cB)$ generated by all $U^{xy}$, $U^{yx}$ (resp. $\Uc_{xy}$, $\Uc_{yx}$), so that $\cB$ is an $R$-$\Rc$-bimodule.
\end{definition}

Note that for each $x \in \cF$, there is an idempotent $\pi_x \in R$ (resp. $\pic_x \in \Rc$) defined by the operator $b \mapsto 1_x \ast b$ (resp. $b \mapsto b\ast 1_x$), which acts on $\cB$ as the projection onto $\cB^x := \bigoplus_y \cB^x_y$ (resp. $\cB_x := \bigoplus_y \cB^y_x$) . 

These idempotents yield decompositions $R = \oplus_{x,y \in \cF} R^{xy}$ and $\Rc = \oplus_{y,z \in \cF} \Rc_{yz}$, where $R^{xy} = \pi_x R \pi_y$ and $\Rc_{yz} = \pic_y \Rc \pic_z$, and the actions of $R$ and $\Rc$ on $\cB$ break up into maps
\[R^{xy} \otimes \cB^y_z \to \cB^x_z\;\; \mbox{and} \;\; \cB^x_y \otimes \Rc_{yz} \to \cB^x_z.\]

\begin{theorem}
The actions of $R$ and $\Rc$ on $\cB$ 
centralize each other; in other words $R = \End_{\Rc}(\cB)$ and
$\Rc = \End_R(\cB)$.  
\end{theorem}

Since the roles of $R$ and $\Rc$ are symmetric, we only need to show that $R = \End_{\Rc}(\cB)$.  Note that by axiom A3 the actions of $R$ and $\Rc$ on $\cB$ commute, so if we put $S := \End_{\Rc}(\cB)$, then $R \subset S$.  
The other inclusion is a  consequence of the following more precise result.  

Note that $\pi_x$ lies in $S$, so $S$ decomposes as $S = \bigoplus_{x,y} S^{xy}$ with $S^{xy} = \pi_xS\pi_y$.  On the other hand, $S$ commutes with the $\pic_z$, so the action of $S^{xy}$ on $\cB$ breaks up into maps
\[S^{xy} \otimes \cB^y_z \to \cB^x_z.\]

\begin{proposition}\label{prop:cellular}
For any $x, y \in \cF$, we have a direct sum decomposition
\begin{equation}\label{Cellular decomp}
S^{xy} = \bigoplus_{\substack{z \le x\\ z\le y}} U^{xz}U^{zy}.
\end{equation}
Furthermore, the multiplication maps $U^{xz} \otimes U^{zy} \to R^{xy}$ are all injective.
\end{proposition}
\begin{proof}
First, to see that $U^{xz} \otimes U^{zy} \to R^{xy}$ is injective, consider the action of $U^{xz} \otimes U^{zy}$ on $\cB^y_z$.  It induces a natural map
 \[\cU^x_z \otimes \cU^y_z \stackrel{\sim}{\to} U^{xz}\otimes U^{zy} \to \Hom_\kk(\cB^y_z, \cB^x_z)\]
sending $u_1 \otimes u_2$ to the map $b \mapsto u_1 \ast (u_2 \ldiv b) = \langle b, u_2\rangle u_1$, which is clearly injective.  Note as well that because $\cU_y^z \subset \cB_y^z$ is saturated, the image of $U^{xz} \otimes U^{zy}$ in $\Hom_\kk(\cB^y_z,\cB^x_z) = (\cB^y_z)^* \otimes \cB^x_z$ is equal to the set of all maps $\phi:\cB^y_z \to \cB^x_z$, such that $\phi(\cB^y_z) \subset \cU^x_z$ and $\phi((\cU^y_z)^\perp)=0$.

To prove that the sum on the right side of \eqref{Cellular decomp} is direct, we need to show that $U^{xz}U^{zy} \cap U^{xz'}U^{z'y} = 0$ for any $z' \neq z$.
Without loss of generality, we may assume that $z \nleq z'$. Note that the triangularity axiom (A1) implies that any element of $U^{xz'}U^{z'y}$ kills $\cB^y_z$.  On the other hand, by the previous remark, $U^{xz}U^{zy}$ maps injectively into $\Hom_\kk(\cB^y_z, \cB^x_z).$  We conclude that the intersection is trivial as was desired.

Now, choose an arbitrary $s\in S^{xy}$.  To see that it lies in the right side of \eqref{Cellular decomp}, choose a total ordering $z_1, \dots, z_n$ of
$\cF$ so that $z_i \le z_j$ implies $i \le j$.  Let $s_n=s$.  We will now inductively construct a sequence of elements in $s_{n-1}, \dots, s_1,s_0 \in S^{xy}$ such that $s_j(\cB^y_{z_i})=0$ for all $i>j$ and $s_{j+1}-s_{j} \in U^{xz_j}U^{z_jy}$.  Observe that if we can construct such a sequence, it will follow that $s \in \oplus_{i} U^{xz_i}U^{z_i y}$.

Suppose we have defined $s_i \in S^{xy}$ with the desired properties for all $i\geq j$.  We argue that there exists $w_j \in U^{xz_j}U^{z_jy}$ such that $(s_j-w_j)(\cB^y_{z_i})=0$ for all $i \geq j$.  

As $s_j$ commutes with the action of $\Rc$, for any $q > z_j$, we have
\[ s_j(\cB^y_q \ast \cUc^q_{z_j}) = s_j(\cB^y_q) \ast \cUc^q_{z_j} = 0. \]
By axiom A2, this means that $s_j((\cU^y_{z_j})^\perp) = 0$.

Similarly, for $q > z_j$ as before, we have
\[ s_j(\cB^y_{z_j})\rdiv \cUc^q_{z_j} = s_j(\cB^y_{z_j}\rdiv \cUc^q_{z_j}) \subset s_j(\cB^y_q) =0. \]
Thus, $\langle s_j(\cB^y_{z_j}), \cB \ast \cUc^q_{z_j} \rangle = \langle s_j(\cB^y_{z_j})\rdiv \cUc^q_{z_j}, \cB \rangle = 0,$
from which we conclude by axiom A2 that $s_j(\cB^y_{z_j}) \subset \cU^x_{z_j}$.

As shown above, any map $\cB^y_{z_j} \to \cB^x_{z_j}$ whose image is contained in $\cU^x_{z_j}$ and kernel contains $(\cU^y_{z_j})^\perp$ arises as the restriction of a unique element of $U^{xz_j}U^{z_j y}$, which we define to be $w_j$.

We then let $s_{j-1}:= s_j-w_j$.  Note that $s_{j-1}$ satisfies the desired properties and our induction is complete.
\end{proof}

\begin{corollary}
For any $x, y \in \cF$, we have a direct sum decomposition
\[
R^{xy} = \bigoplus_{\substack{z \le x, y}} U^{xz}U^{zy} \;\; \mbox{and} \;\; \Rc_{xy} = \bigoplus_{\substack{z \ge x, y}} \Uc_{xz}\Uc_{zy}.
\]
\end{corollary}

We will use the above description to prove that $R$ and $\Rc$ are quasi-hereditary.  First, we recall the definition of cellular algebras and their relation to quasi-hereditary algebras.  Cellular algebras were introduced by Graham and Lehrer~\cite{GrLe} in terms of the behaviour of a nice basis.  K\"onig and Xi~\cite[Section 3]{KXstruc} have since shown that the following is an equivalent basis-free definition.

\begin{definition}
Let $A$ be an algebra over $k$ endowed with an involution $i$ that is an anti-automorphism.
\begin{enumerate}
\item A two-sided ideal $J \subset A$ is said to be a \textbf{cell ideal} if there exists a left ideal $\Delta \subset A$ such that $\Delta$ is a finitely generated and free as a $k$-module and there exists an isomorphism of $A$-bimodules $\alpha \colon J \cong \Delta \otimes i(\Delta)$ for which the following diagram is commutative:
\[
\xymatrix{
J \ar[rr]^\alpha \ar[d]_i && \Delta \otimes i(\Delta) \ar[d]^{x\otimes y \mapsto i(y) \otimes i(x)} \\
J \ar[rr]^\alpha && \Delta \otimes i(\Delta).
}
\]
\item $A$ is said to be \textbf{cellular} if there is a decomposition of $k$-modules $A = J'_1 \oplus \dots \oplus J'_n$ such that $i(J'_j) = J'_j$ and $J_j := \oplus_{l=1}^j J'_l$ is a two-sided ideal for all $j$ and the quotient $J'_j = J_j/J_{j-1} \subset A/J_{j-1}$ is a cell ideal.  In this case, the chain of ideals $0 \subset J_1 \subset \dots \subset J_n = A$ is called a \textbf{cellular chain}.
\end{enumerate}
\end{definition}

K\"onig and Xi also proved the following:

\begin{lemma} \label{lemma:cellqh}
\cite[Lemma 2.1(3)]{KXwhen} 
If a cellular algebra over a field $\mathbb{F}$ has a cellular chain of length $n$, then the number of isomorphism classes of simple $R$-modules is bounded by $n$.  Moreover, equality holds if and only if $A$ is quasi-hereditary with heredity chain given by the cellular chain but with the ideals indexed in reverse order (cf. the warning following \cite[Cor. 4.2]{KXstruc}) and the standard modules $\Delta_x$ for $A$ are just the cell modules.
\end{lemma}

With these preliminaries, we may now show:

\begin{theorem}
$(\Rc, \cF)$ and $(R, \cF^{opp})$ are quasi-hereditary algebras.
\end{theorem}

Here $\cF^{opp}$ denotes the opposite poset of $\cF$.

\begin{proof}
As $R$ and $\Rc$ are defined symmetrically, it suffices to show the statement for $R$.  We proceed by first exhibiting a cellular chain for $R$.

For any $r \in R \subset \End(\cB)$, let $i(r) \in R$ be the adjoint of $r$.  Note that the involution $i$ is an antiautomorphism.

Recall the total ordering $z_1, \dots,z_n$ of $\cF$ from the proof of Proposition~\ref{prop:cellular}.  For each $j \in \cF$, let $J'_j := \bigoplus_{x,y \geq z_j} U^{xz_j}U^{z_j y}$.  Note that $J'_j$ is preserved by adjunction and by Proposition~\ref{prop:cellular}, the $J'_j$ intersect trivially and $R = J'_1 \oplus \dots \oplus J'_n$.  

Let $J_j := \bigoplus_{i \leq j} J'_i = \bigoplus_{i \leq j} \bigoplus_{z_i \leq x,y} U^{xz_i}U^{z_i y}$.  Note that it can be reinterpreted as $\{r \in R \mid r (\cB_{z_k}) =0, \forall k>j\}$, where $\cB_{z_k} = \bigoplus_x \cB^x_{z_k}$.  From this latter description and the invariance of $J_j$ under the involution $i$, it follows that $J_j$ is a two-sided ideal.  Let $\Delta_{z} := \bigoplus_{x \geq z} U^{xz}$.  By Lemma~\ref{lemma-A4}, $\Delta_j$ is a left ideal of $R/J_{j-1}$.  By definition $i(U^{xy}) = U^{yx}$ and there is a canonical isomorphism $J'_j = \Delta_{z_j} \otimes i(\Delta_{z_j})$ making the diagram in the definition of a cell ideal commute.  Thus $R$ is cellular.

Note that the cell chain is free over $k$, so by~\cite[Theorem 3.3]{CPSint}, to prove $R$ is quasi-hereditary it suffices to show that for any $\mathfrak{p} \in \mathrm{Spec}(k)$, the cellular chain $J_j \otimes \mathbb F$ is a heredity chain for $A \otimes \mathbb F$, where $\mathbb F$ is defined as $k_{\mathfrak p}/{\mathfrak p} k_{\mathfrak p}$.

By Lemma~\ref{lemma:cellqh}, it remains to show that for any residue field $\mathbb F$ of $k$, there is one simple object of $A \otimes \mathbb F$ for each idempotent $\pi_x$.  Define the support of an $R$-module $M$ to be the set of $x \in \cF$ for which $\pi_x M \ne 0$.  Then the support of $\cB_y = \bigoplus_x \cB^x_y$ contains $y$ and is contained in the set of $x$ with $x \ge y$.  It follows that for each $y$ there is a simple module so that $y$ is the minimal element in its support.
\end{proof}

\begin{corollary}\label{cor:standard module}
For any $x \in \cF$, the standard module $\Delta_x$ for $R$ is isomorphic to \[\bigoplus_{z \geq x} U^{zx} = R\cB_x^x.\]
\end{corollary}

\begin{theorem}
$\cB$ is tilting as an $R$-module and as an $\Rc$-module.
\end{theorem}
\begin{proof}
It is enough to show that $\cB$ is tilting as an $R$-module, and since we have a decomposition 
$\cB = \bigoplus_{x\in \cF} \cB_x$ of $R$-modules, it is enough to show each $\cB_x$ is tilting.  

Since the actions of $U^{xy}$ and $U^{yx}$ are adjoint under the pairing, the $R$-module $\cB_x$ is self-dual.  It is therefore enough to show that it has a filtration by standard modules.  

Consider the filtration of $\cB_x$ by the submodules
$R\cB^{\le y}_x$ for all $y\ge x$, where $\cB^{\le y}_x := \bigoplus_{z \le y} \cB^z_x$.  
If $\check{u} \in \cUc^y_{x}$ the map
\[\cdot \rdiv \check{u}\colon R\cB^{\le y}_x \to
\cB_y\]
 is an $R$-module homomorphism by our associativity 
axiom A3, and by 
the triangularity axiom A1
it annihilates the submodule $R\cB^{<y}_x \subset R\cB^{\le y}_x$ and its image lies in
 $R\cB^y_y \cong \Delta_y$.  
 Putting these maps together over all $\check{u}$, we get
an $R$-module map 
\begin{equation} \label{standard layer}
(R\cB^{\le y}_x)/(R\cB^{<y}_x) \to \Delta_y \otimes (\cUc^y_x)^*. 
\end{equation}
Applying $\pi_y$ to both sides gives an isomorphism, 
since the left hand side is 
\[\cB^y_x\left/ \left(\sum_{z \ne y} \cU^y_z \cB^z_x\right)\right. \cong \cB^y_x/(\cUc^y_x)^\bot \cong (\cUc^y_x)^*,\]
by axiom A2.
It follows that this map is surjective.
On the other hand, 
since the left-hand side of \eqref{standard layer} is generated by its 
component at $y$, and the component at $z$ vanishes for all $z < y$,
the left-hand side is isomorphic to a quotient of $\Delta_y \otimes (\cUc^y_x)^*$.  It follows that \eqref{standard layer} is an isomorphism, 
completing the proof.
\end{proof}

\section{Ringel construction for matroids} \label{sec:matroid}

Let $M$ be a matroid with underlying set $I$ of $n$ elements and $\cI$ the collection of independent subsets.  We will fix once and for all an ordering\footnote{The choice of an ordering is not used in the following construction, but will be useful at various points in our proofs.} on the set $I$.
The other input we will need for our construction is the choice of a weight function $a : I \to k^*$.  In this section we associate a Ringel datum to any matroid endowed with a weight function.

\subsection{Matroids: notation and background}

Recall that a maximal independent set $B \in \cI$ is called a \textsl{basis} of $M$ and that all bases of $M$ have the same number of elements.  The \textsl{rank} $r=\rho(M)$ of $M$ is defined to be the size of a basis for $M$.  Let $\Bas(M)$ denote the set of all bases of $M$.

For any $X \subset I$, let $\cI|X$ be $\{E \subset X  \mid E \in \cI\}$.  The resulting collection of subsets of $X$ forms a matroid $M(X)$ called the \textsl{restriction} of $M$ to $X$.  The \textsl{rank} $\rho(X) = \rho(M(X))$ of $X$ is the size of any basis of $M(X)$.

The \textsl{closure} or span of a subset $X \subset I$ is defined as 
\[\overline{X} = \{x \in I \mid \rho(X \cup \{x\}) = \rho(X)\}.\]

A subset $X \subset I$ is called a \textsl{flat} if $X = \overline{X}$.  If $X$ and $Y$ are flats, then $X \cap Y$ is a flat, and so is $X \vee Y := \overline{X \cup Y}$.

An element $x \in I$ is called a \textsl{loop} of $M$ if the subset $\{x\}$ is dependent and is called a \textsl{coloop} of $M$ if $x$ is an element of every basis.

Let $\cF$ be the poset of coloop-free flats of $M$, i.e. flats $X$ such that $M(X)$ has no coloops.  The order on this poset is by \emph{inverse} inclusion; this non-standard choice comes from the relation of this construction with hyperplane arrangements and the topology of hypertoric varieties.  Let $\bm{1} \in \cF$ denote the unique maximal element of $\cF$ and $\bm{0} \in \cF$ denote the unique minimal element.  Note that $\bm{1}$ is equal to the flat consisting of all loops and $\bm{0}$ is the complement to the set of coloops.

We now recall the notion of matroid duality.  For any matroid $M$, let $\Bas^*(M)$ denote the collection of subsets $\{I \setminus B \subset I\mid  B \in \Bas(M)\}$.  The set $\Bas^*(M)$ is the set of bases for the \textsl{dual matroid} $M^*$, which is also defined on the underlying set $I$.

For $E \subset I$, the \textsl{contraction} $M/E$ of $M$ by $E$ is defined as $M/E = (M^*(I\setminus E))^*$.

We will also use two invariants associated to any flat $K$: Crapo's beta invariant $\beta(K)$ and the unsigned M\"obius function $\mu^+(K)$ (see, e.g.,~\cite[page 4]{BV} for definitions).

\subsection{The ring $\cB$}

Let $\Lambda(I) = \bigoplus_p \Lambda^p(I)$ be the exterior algebra with base ring $k$ over the set $I$, so 
$\Lambda^1(I)$ has the elements of $I$ as a basis.  For any subset $S = \{s_1, \dots,s_p\} \subset I$ ordered so that $s_i < s_{i+1}$, let $e_S = s_1 \wedge \dots \wedge s_p \in \Lambda(I)$.  The set of monomials $e_S$ over all subsets $S$ forms a basis of $\Lambda(I)$.

Let $\Lambda^p_q(M) \subset \Lambda^p(I)$ be the free submodule generated by monomials $e_S$ such that $S\subset I$ is of rank $q$ in $M$.

\begin{definition} Let $\cB(M) := \Lambda^r_r(M)$.  For any $E, F \in \cF$ such that $E\geq F$, let $\cB^E_F = \cB(M(F)/E)$.  For $E \not\geq F$, we let $\cB^E_F =0$.

If $E \geq F \geq G$, then we define multiplication $\ast: \cB^E_F \otimes \cB^F_G \to \cB^E_G$ by the wedge product.  This makes sense because the union of a basis of $M(F)/E$ and a basis of $M(G)/F$ is a basis of $M(G)/E$.
\end{definition}

Lastly, we endow $\Lambda(I)$ with the symmetric perfect pairing $\langle,\rangle$ defined as follows on the monomial basis:
\[
\la e_S, e_T \ra = \begin{cases}
    \prod_{s \in S} a(s)^{-1} & \text{if $S=T$}, \\
    0 & \text{otherwise}.
  \end{cases}
\]
In particular, this restricts to a symmetric perfect pairing on $\cB$.

To ease notation, we will let $a(S) = \prod_{s \in S} a(s)$.

Similarly to the general framework, we will consider adjoint operations to the wedge product in the exterior algebra.  Let $\lcon$ and $\rcon$ denote the left and right adjoints.  Note that if $b \in \cB^G_F$ and $b' \in \cB^G_E$, then $b \lcon b' = b \ldiv b' \in \cB^F_E$.

We observe that the wedge product and contraction are compatible with the bracket in the following ways.

\begin{lemma}
Fix any subset $S \subset I$.  Suppose $x,x' \in \Lambda(S)$ and $y,y' \in \Lambda(I \setminus S)$.  Then
\[\la x \wedge y , x' \wedge y' \ra = \la x,x' \ra \la y,y'\ra.\]
\end{lemma}

\begin{proof}
It suffices to check when $x,x',y,y'$ are all monomials.  In this case it follows directly from the definition of the pairing.
\end{proof}

\begin{corollary}\label{cor-distr}
Fix any subset $S \subset I$.  Suppose $x,x' \in \Lambda(S)$ and $y,y' \in \Lambda(I \setminus S)$ and that each is of homogeneous degree.  Then
\[ (x \wedge y) \lcon (x' \wedge y') = (-1)^{(|x|-|x'|)|y|} (x \lcon x') \wedge (y\lcon y'). \]
\end{corollary}

\begin{proof}
For every monomial $e_a \in \Lambda(S)$ and $e_b \in \Lambda(I\setminus S)$ consider the pairing:
\[ 
\begin{split}
\la (x \wedge y) \lcon (x' \wedge y'), e_a \wedge e_b \ra & =\la x' \wedge y', x \wedge y \wedge e_a \wedge e_b \ra \\
& = \la x' \wedge y', (-1)^{|e_a||y|} (x \wedge e_a) \wedge (y \wedge e_b) \ra \\
& = (-1)^{|e_a||y|} \la x', x \wedge e_a\ra \la y', y \wedge e_b) \ra \\
& = (-1)^{|e_a||y|} \la x\lcon x', e_a\ra \la y \lcon y', e_b) \ra 
\end{split}
\]
Lastly, notice that $\la x\lcon x', e_a\ra =0$ unless $|e_a|=|x'|-|x|$.
\end{proof}

\subsection{The subspaces $\cU$ and $\cUc$}
We now define the $k$-submodules $\cU(M), \cUc(M) \subset \cB(M)$ and $\cU^E_F, \cUc^E_F \subset \cB^E_F$ for any $E,F \in \cF$.

Our definition of $\cU(M)$ and $\cUc(M)$ is based on the following bicomplex structure on the exterior algebra $\Lambda(M)$ introduced by Denham~\cite{Denham}.  Let $\partial:\Lambda(I) \to \Lambda(I)$ be the standard differential on the exterior algebra, defined on monomials $e_S \in \Lambda^p(I)$ by
\[ \partial(e_S) = \sum_{i =1}^p (-1)^i s_1 \wedge \dots \wedge \hat{s_i} \wedge \dots \wedge s_p.\]
Note that $\partial$ can be expressed as a sum of two differentials $\partial_h$ and $\partial_v$ where $\partial_h: \Lambda^q_p(M) \to \Lambda^{q-1}_{p-1}(M)$ and  $\partial_v: \Lambda^q_p(M) \to \Lambda^{q-1}_{p}(M)$ are obtained by the composition of $\partial:  \Lambda^q_p(M) \to \Lambda^{q-1}_{p-1}(M) \oplus \Lambda^{q-1}_{p}(M)$ with the two projections.

Let $\delta,\delta_h,$ and $\delta_v$ be the adjoints with respect to $\langle,\rangle$ of $\partial,\partial_h,$ and $\partial_v$, respectively.  Note that the restriction of $\delta$ to $\Lambda^r_r(M)$ is equal to the restriction of $\delta_v$ and similarly for $\partial$ and $\partial_h$.  Observe as well that for any $e \in \Lambda(I)$, $\delta$ is given by:
\[ \delta(e) = \sum_{s \in I} a(s) s \wedge e.\]

\begin{definition}
Let $\cU(M) = \Ker (\partial: \Lambda^r_r(M) \to \Lambda^{r-1}_{r-1}(M))$ and $\cUc(M) = \Ker(\delta: \Lambda^r_r(M) \to \Lambda^{r+1}_r(M)).$

For any $E,F \in \cF$ such that $E \geq F$, let $\cU^E_F = \cU(M(F)/E)$ and $\cUc^E_F = \cUc(M(F)/E)$.
\end{definition}

These subspaces are interchanged by matroid duality in the following sense.

\begin{definition}
Let $\D : \Lambda(I) \to \Lambda(I)$ be the linear map $v \mapsto v \lcon e_I$.
\end{definition}

For any $S,T \subset I$, let $\epsilon(S,T)$ be $(-1)^k$, where $k$ is the number of pairs $(s,t) \in S \times T$ such that $s > t$.

\begin{lemma}
The map $\D$ is an isomorphism that takes $\Lambda^k(I)$ to $\Lambda^{n-k}(I)$ and $\cB(M) =\Lambda^r_r(M)$ to $\cB(M^*) = \Lambda^{n-r}_{n-r}(M^*)$.
\end{lemma}

\begin{proof}
We simply compute $\D$ on the monomial generators.  Suppose $S \subset I$ and $e_S \in \Lambda^k(I)$.  We observe 
\[ \la \D(e_S),e_T \ra = \la e_I , e_S \wedge e_T \ra = \begin{cases}
     \epsilon(S,I\setminus S) a(I)^{-1} & \text{if $T=I \setminus S$}, \\
     0 & \text{otherwise}
\end{cases}
\]
Hence $\D(e_S) = \epsilon(S,I \setminus S) a(S)^{-1} e_{I\setminus S} \in \Lambda^{n-k}(M)$.  Repeating this argument we find that 
\[ \D^2(e_S) = \epsilon(S,I\setminus S) a(S)^{-1} \D(e_{I\setminus S}) =  (-1)^{k(n-k)} a(I)^{-1} e_S.\]
\end{proof}

\begin{lemma}
\label{lem:matdual}
The duality map $\D$ enjoys the following properties:
\begin{enumerate} 
\item For any $v \in \Lambda^k(M)$, 
\[\delta (\D(v)) = (-1)^{k+1} \D (\partial(v)) ~\mathrm{and}~ \partial (\D(v)) = (-1)^{k} \D (\delta(v)),\]
\item $\D$ maps $\cU(M)$ (resp. $\cUc(M)$) isomorphically to $\cUc(M^*)$ (resp. $\cU(M^*)$),
\item The adjoint of $\D|_{\cB(M)}$ is $(-1)^{r(n-r)}\D|_{\cB(M^*)}$,
\item For any flat $K$ of $M$ of rank $r'$, $x\in \cB(M(K))$ and $y \in \cB(M/K)$,
\[ \D(x \wedge y) = (-1)^{r'(r-r')} \epsilon(K, I \setminus K) \D^K(y) \wedge \D_K(x),\]
where $\D^K$ (resp. $\D_K$) denotes the duality map for the matroid $M/K$ (resp. $M(K)$).
\end{enumerate}
\end{lemma}

\begin{proof}

For the first statement of part (1), we wish to show that for $v \in \Lambda^r$, $\delta(\D(v)) = (-1)^{r+1} \D(\partial(v))$.  We do so by considering the pairing with any $w \in \Lambda^{n-r+1}$.  By adjunctions we have:
\[ \la \delta(\D(v)), w \ra = \la v \lcon e_I, \partial w \ra = \la e_I, v \wedge \partial w \ra. \]
As $\partial$ is a derivation, $\partial(v \wedge w) = \partial(v) \wedge w + (-1)^r v \wedge \partial(w)$.  But $v\wedge w$ has degree $n+1$ and hence vanishes.  We conclude that the above quantity is equal to
\[\la e_I, (-1)^{r+1}\partial(v) \wedge w \ra = \la (-1)^{r+1}\D(\partial v), w \ra.\]

A similar argument yields the second statement.  Alternatively, the second statement can also be obtained from the first by pre- and post-composing with the duality operator and using the previous lemma.

For part (2), note that by part (1) and the previous lemma, $\delta(\D(v)) = 0 $ if and only if $\D(\partial(v)) =0$ if and only if $\partial(v)=0$.  Similarly we find $\partial(\D(v))=0$ if and only if $\D(\delta(v))=0$ and hence if and only if $\delta(v)=0$.

Part (3) follows easily from the definitions:
\[ \la \D v, w \ra = \la e_I, v \wedge w \ra = (-1)^{r(n-r)}\la w \wedge v , e_I \ra = (-1)^{r(n-r)}\la v , \D w \ra. \]

Lastly, we use Corollary~\ref{cor-distr} to prove part (4).  We have $e_I = \epsilon(K,I\setminus K) e_K \wedge e_{I \setminus K}$, thus
\[ 
\begin{split}
\D( x\wedge y) & = \epsilon(K, I \setminus K) (x \wedge y) \lcon (e_K \wedge e_{I \setminus K}) \\
&  = \epsilon(K, I \setminus K) \D_K(x) \wedge \D^K(y) \\
& = (-1)^{r'(r-r')} \epsilon(K, I \setminus K) \D^K(y) \wedge \D_K(x).
\end{split}
\]
\end{proof}

\subsection{Dimension formulas for $\cU$ and $\cUc$}

\begin{lemma}
\label{lem:dim}
 $\cU(M)$ and  $\cUc(M)$ are free, saturated $k$-submodules of $\cB(M)$ and respectively of ranks $ \mu^+(M^*)$ and $\mu^+(M)$.
\end{lemma}

\begin{proof}
Note that $\cU(M)= \Ker(\partial) = \wt{H}_{r-1}(IN(M))$, the top reduced homology of the independence complex.  Recall from~\cite[Theorem 7.8.1]{Bj92} that $\wt{H}_{k}(IN(M;\Z)) = \Z^{\mu^+(M^*)}$ if $k=r-1$ and $0$ otherwise.  By the universal coefficients theorem, we deduce that $\cU(M)$ is a free $k$-module of rank $\mu^+(M^*)$ and by duality $\cUc(M)$ is also free and of rank equal to $\rk \cU(M^*) = \mu^+(M)$.  As $k$ is a domain and $\partial$ and $\delta$ are $k$-linear maps, $\cU(M)$ and $\cUc(M)$ are both saturated submodules.
\end{proof}

It will also be useful to have at our disposal the notions of externally and internally passive bases with respect to our chosen order on the set $I$.  Recall that if $B$ is a basis of $M$, then for any $p \in I \setminus B$, there is a unique circuit $\ci(B,p)$ contained in $B \cup p$.  Dually, for any $b \in B$ there is a unique bond $\bo(B,b)$ in $(I\setminus B) \cup b$.   An element $p \in I \setminus B$ is said to be externally active if $p$ is the minimal element in the basic circuit $\ci(B,p)$, otherwise $p$ is said to be externally passive.  An element $p \in B$ is said to be internally active if $p$ is minimal in the basic bond $\bo(B,p)$, otherwise $p$ is said to be internally passive.

We say a basis $B$ is externally (resp. internally) passive if every element of $I \setminus B$ (resp. $B$) is externally (resp. internally) passive.  An externally passive basis is also referred to as an \textit{nbc}-basis.

We conclude this paragraph by recording the following well-known equality.

\begin{proposition}
\label{prop:pass}
The number of externally (resp. internally) passive bases of $M$ is equal to $\mu^+(M)$ (resp. $\mu^+(M^*)$).
\end{proposition}

\begin{proof}
The statement about externally passive bases is a special case of~\cite[Proposition 7.4.5]{Bj92}.  As duality exchanges internally and externally passive bases, the other statement follows.
\end{proof}


\subsection{Proof of Axiom A2}

Assuming that $\cU^\perp = \cB \ast \cUc^+$, we can deduce that $\cUc^\perp = \cU^+ \ast \cB$ as follows:
\[
\begin{split}
\cUc(M^E_F)^\perp & = (\D_{F \setminus E} \cU((M^*)^{I \setminus F}_{I\setminus E})^\perp \\
& = \D_{F\setminus E}((\cU(M^*)^{I \setminus F}_{I\setminus E})^\perp) \\
& = \D_{F\setminus E}(\bigoplus_{G \in \cF, F \leq G < E} \cB((M^*)^{I\setminus F}_{I\setminus G}) \ast \cUc((M^*)^{I \setminus G}_{I\setminus E})) \\
& = \bigoplus_{G \in \cF, F \leq G < E} (\D_{G\setminus E} \cUc((M^*)^{I \setminus G}_{I\setminus E})) \ast (\D_{F\setminus G} \cB((M^*)^{I\setminus F}_{I\setminus G})) \\
& = \bigoplus_{G \in \cF, F \leq G < E} \cU(M^E_G) \ast \cB(M^G_F).
\end{split}
\]
Here we use Lemma~\ref{lem:matdual}: part (2) for the first and last equality, (3) for the second and (4) for the fourth.

Thus it suffices to show that $\cU^\perp = \cB \ast \cUc^+$.  Here we may assume that $M$ is coloop-free.  We proceed by proving the following three relations:  $\cU^\perp = \Im(\delta_h: \Lambda^{r-1}_{r-1}(M) \to \Lambda^r_r(M))$, $\Im(\delta_h) \subset \cB \ast \cUc^+$, and $\cB \ast \cUc^+ \subset \cU^\perp$.

\begin{proposition} \label{prop-Uperp}
$\cU^\perp = \Im(\delta_h: \Lambda^{r-1}_{r-1}(M) \to \Lambda^r_r(M))$.
\end{proposition}

\begin{proof}
By adjunction, $\Im(\delta_h) \subset \cU^\perp$. We have seen, by Lemma~\ref{lem:dim} and Proposition~\ref{prop:pass}, that $\cU(M) \subset \cB(M)$ is saturated and free of rank equal to the cardinality of the set $\ol\Bas$ of internally passive bases.  Thus the saturated submodule $(\cU(M))^\perp \subset \cB(M)$ has complementary rank.  It thus suffices to show that the saturated submodule $\Im(\delta_h) \subset (\cU(M))^\perp$ has the same rank, which would imply equality.  To do so, we will show that $\Im(\delta_h)+\ol{\Bas} = \Bas$.

We show by downward induction on the lexicographic order on bases $B$ that every monomial $e_B$ is contained in $\Im(\delta_h)+\ol{\Bas}$.  Suppose that $e_{B'} \in \Im(\delta_v)+\ol{\Bas}$ for all $B' >B$.  If $B$ is internally passive then certainly $e_{B} \in \Im(\delta_h)+\ol{\Bas}$.  If $B$ is not internally passive, then by definition there exists an element $p \in B$ such that for any $j < p$, $(B \setminus \{p\}) \cup \{j\}$ is not a basis.  We conclude that 
\[\delta_h(e_{B \setminus \{p\}}) \in \pm a(p) e_B + \mathrm{Span}(e_{B'}| B' >B), \]
which completes our induction.
\end{proof}

\begin{proposition}
$\Im(\delta_h) \subset \cB \ast \cUc^+$.
\end{proposition}

\begin{proof}
For any independent set $S \subset I$ of cardinality $r-1$, we wish to show that $\delta_h(e_S) \in \cB \ast \cUc^+$.  Consider the set
\[J = \{j \in I \setminus S \mid S \cup \{j\} \; \text{is not a basis}\}.\]
For each $j\in J$, there is a unique circuit $C_j$ contained in $S \cup \{j\}$.  Let $F$ denote the flat $\bigcup_{j \in J} C_j \subset I$.  Note that $F\neq I$ and is a coloop-free flat with basis $T:= S \cap F$ such that $J \subset F \subset J \cup S$.

Simple manipulation of symbols yields:
\begin{equation*}
\begin{split}
 \delta_h e_S & =  \sum_{i \in I \setminus J} a(i) e_i \wedge e_S = \sum_{i \in I \setminus F} a(i) e_i \wedge e_S \\ 
 & = \epsilon(T, S \setminus T) \sum_{i \in I \setminus F} a(i) e_i \wedge e_T \wedge e_{S \setminus T} \\
& = (-1)^{|T|} \epsilon(S \setminus T, T) e_T \wedge \sum_{i \in I \setminus F} a(i) e_i \wedge e_{S\setminus T}\\
& = (-1)^{|T|} \epsilon(S \setminus T, T) e_T \wedge \delta^F (e_{S \setminus T}),
\end{split}
\end{equation*}
where $\delta^F$ denotes the boundary operator $\delta$ for the contracted matroid $M / F$.  Note that for any $i \in I \setminus F$, $(S\setminus T) \cup \{i\} \subset S \cup \{i\}$ is independent and so $\delta^F  (e_{S \setminus T}) = \delta^F_h (e_{S \setminus T})$.  It follows that $\delta^F(e_{S \setminus T}) \in \Ker(\delta^F) = \cUc^F_{I}$ and hence we have shown that $\delta_h e_S \in \cB^{\emptyset}_F \ast \cUc^F_{I}$.
\end{proof}

\begin{proposition}\label{prop:B times Uc perp U}
We have $\cB \ast \cUc^+ \subset \cU^\perp$.  More generally, for any (not necessarily coloop-free) flat $K \neq I$,
\[ \cB(M(K)) \wedge \cUc(M/K) \subset \cU(M)^\perp \]
and for any flat $K \ne \emptyset$
\[ \cU(M(K)) \wedge \cB(M/K) \subset \cUc(M)^\perp, \]
where $\cU(M)^\perp$ and $\cUc(M)^\perp$ denote the perpendicular subspaces in $\cB(M)$. 
\end{proposition}

\begin{proof}  By the same duality argument given above, it suffices to prove the first inclusion.
We begin with the following special case.

\begin{lemma}\label{lem-orthog}
Assume that $M$ is a non-empty matroid (not necessarily coloop-free).  Then $\cU(M)$ and $\cUc(M)$ are orthogonal to each other.
\end{lemma}

\begin{proof}
Let $P = \Z[b_s,b_s^{-1}]_{s \in I}$ be the ring of Laurent polynomials in $I$.  In this proof we use the universal weight function $\tilde a: I \to P$ which takes $s \in I$ to $b_s$ to define versions of $\cB$, $\partial$, $\delta$, $\cU$ and $\cUc$ over $P$.  Note that after changing rings from $P$ back to $k$, the modified bilinear form and operators base change to the original ones.

Consider as in ~\cite{Denham} the operators $\Delta_h := \delta_h \circ \partial_h + \partial_h \circ \delta_h$ and $\Delta_v := \delta_v \circ \partial_v + \partial_v \circ \delta_v$.

Note that $\Delta_h(\cU(M)) = 0$.  Similarly, $\Delta_v(\cUc(M))=0$.  A simple computation shows that 
\[\Delta_h + \Delta_v |_{\cB(M)}= (\sum_{s \in I} b_s) \id_{\cB(M)}.\]
 Hence $\Delta_h |_{\cUc(M)} = (\sum_{s\in I} b_s) \id_{\cUc(M)}$.  We conclude that $\cU(M)$ and $\cUc(M)$ are eigenspaces for the self-adjoint operator $\Delta_h$ with distinct eigenvalues and thus orthogonal.  Base changing to $k$ preserves this orthogonality and we obtain the desired result.
\end{proof}

To show that $\cB(M(K)) \wedge \cUc(M/K) \subset \cU(M)^\perp$, consider any $x \in \cB(M(K))$, $\check{u} \in \cUc(M/K)$ with $K \ne I$, and any $u \in \cU(M)$. We wish to show that $\la x \wedge \check{u}, u\ra = 0$, or equivalently that $\la \check{u}, x \lcon u \ra = 0$.  By the previous lemma, it is enough to show that the contraction $x \lcon u \in \cB(M/K)$ is in $\cU(M/K)$ or, in other words, that $\partial^K(x \lcon u)=0$.  Here $\partial^K$ (resp. $\delta^K$) denotes the differential $\partial$ (resp. $\delta$) for the matroid $M/K$ and similarly $\partial_K$ and $\delta_K$ denote the differentials for $M(K)$.


It suffices to consider the case when $x=e_S$ for a basis $S$ of $K$.  This is the content of the following lemma, whose proof completes the proof of the proposition.

\begin{lemma} \label{lem-div}
Let $S \in \cI$ be an independent set and $u \in \cU(M)$.  Then the contraction $e_S \lcon u$ is an element of $\cU(M/\overline{S})$.
\end{lemma}

Let $r' = \rho(M/\overline{S})$. Then $\partial^{\overline{S}}(e_S \lcon u) \in \Lambda^{r'-1}_{r'-1}(M/\overline{S})$.  To show that $\partial^{\overline{S}}(e_S \lcon u)=0$ it thus suffices to show that for any $w \in \Lambda^{r'-1}_{r'-1}(M/\overline{S})$, the pairing $\la \partial^{\overline{S}}(e_S \lcon u), w \ra$ vanishes.  By adjunction,
\[ \la \partial^{\overline{S}}(e_S \lcon u), w \ra = \la u, e_S \wedge \delta^{\overline{S}} w \ra.\]
The Leibniz rule for $\delta$ tells us that $\delta(e_S \wedge w)= \delta_{\overline{S}}(e_S) \wedge w \pm e_S \wedge \delta^{\overline{S}} w$.  Thus the above expression can be rewritten as
\[ = \pm(\la u, \delta(e_S \wedge w) \ra - \la u, \delta_{\overline{S}}(e_S) \wedge w \ra) \]
By adjunction and the definition of $\cU$, the first term vanishes and we are left with
\[ = \pm \la u, \delta_{\overline{S}}(e_S) \wedge w \ra.\]
But $e_S \in \cB(M(\overline{S})) = \Lambda^{r-r'}_{r-r'}(M(\overline{S}))$, so $\delta_{\overline{S}}(e_S) \in \Lambda_{r-r'}^{r-r'+1}(M(\overline{S}))$.
We conclude that the wedge product $\delta_{\overline{S}}(e_S) \wedge w$ lies in $\Lambda_{r-1}^r(M)$,\footnote{Warning: here we use that if $T\subset E$ is dependent in $M(E)$, then for any $U \subset I \setminus E$, $T \cup U$ is dependent in $M$.  Note however, that if $T' \subset I \setminus E$ is dependent in $M/E$ and $U' \subset E$, then $T' \cup U'$ need NOT be dependent in $M$.} while $u \in \Lambda^r_r(M)$ and so the pairing vanishes. \end{proof}



\subsection{Proof of Axiom A3}

Since the two equations of Axiom A3 are equivalent by adjunction, we only need to prove the first one.  Without loss of generality we can assume the largest and smallest flats that occur in this computation are $\bm{1}$ and $\bm{0}$, respectively.  We wish to show for any $E,F \in \cF$, any basis $B$ of $M^{\bm{1}}_E$ and any elements $u \in \cU^{\bm{1}}_F, \check{u} \in \cUc^E_{\bm{0}}$ there is an equality:
\[(u \ldiv e_B)\ast \check{u}  = u \ldiv (e_B \ast \check{u}).\]

We consider two cases: either $F \geq E$ or it is not.

If $F \geq E$, then $u \in \Lambda(E)$ and by Corollary~\ref{cor-distr},
\[
\begin{split} 
(u \ldiv e_B) \ast \check{u} & = (u \lcon e_B)\wedge \check{u} \\
& = (u \wedge 1) \lcon (e_B \wedge \check{u}) \\
& = u \lcon (e_B \wedge \check{u}) = u \ldiv (e_B \ast \check{u}).
\end{split}
\]

If $F \ngeq E$, then $u\ldiv e_B = 0$, so the left hand side vanishes.  It thus remains to show that the right hand side also vanishes.  To check that it does, it suffices to show that for any $b' \in \cB^F_{\bm{0}}$, the pairing $\la u \ldiv (e_B \ast \check{u}), b' \ra = \la u \lcon (e_B \wedge \check{u}), b' \ra$ vanishes.  By adjunction:
\[\la u \lcon (e_B \wedge \check{u}), b' \ra = \la \check{u}, e_B \lcon (u \wedge b') \ra.\]
Let $S = B \cap F$ and $T = B \setminus S$. Then $e_B = \pm e_{S} \wedge e_{T}$. Thus
\[ 
\begin{split}
\la \check{u}, e_B \lcon (u \wedge b') \ra  & = \pm \la \check{u}, (e_{S} \wedge e_{T}) \lcon (u \wedge b') \ra \\
& = \pm \la  \check{u}, (e_S \lcon u) \wedge (e_T \lcon b') \ra. 
\end{split}
\]

There are now two possibilities, either $T$ is an independent set in $M(E\vee F)/F$ or it is not.

If $T$ is not independent in $M(E\vee F)/F$, then $e_T \lcon b'=0$ and we are done.

We will thus assume that $T$ is independent in $M(E\vee F)/F$, in which case
\[ |T| \leq \rk(E\vee F) - \rk(F).\]
On the other hand, $|T| = \rk(E)-|S| \geq \rk(E) - \rk(E \cap F)$.  Using these observations, we find that the standard inequality
\[ \rk(E) + \rk(F) \geq \rk(E \vee F) + \rk(E\cap F)\]
is in fact an equality.  It follows that $T$ is a basis of $M(E\vee F)/F$ and $S$ is a basis of $M(E\cap F)$. 

Consider $e_S \lcon u$.  The previous paragraph implies $e_S \in \cB^{\bm{1}}_{E\cap F}$ and by Lemma~\ref{lem-div}, $e_S \lcon u \in \cU^{E\cap F}_F$.  Moreover the rank equality implies that any basis of $M(F)/(E\cap F)$ is also a basis of $M(E\vee F)/E$, so there is a natural inclusion $\cU^{E\cap F}_F \subset \cU^{E}_{E \vee F}$ and $e_S \lcon u \in \cU^{E}_{E \vee F}$.

Because $T$ is a basis of $M(E\vee F)/F$, we find that $e_T \lcon b' \in \cB^{E \vee F}_{\bm{0}}$.

We conclude that 
\[(e_S \lcon u) \wedge (e_T \lcon b') = (e_S \lcon u) \ast (e_T \lcon b') \in \cU^F_{E\vee F} \ast \cB^{E\vee F}_{\bm{0}}.\]

By Proposition~\ref{prop:B times Uc perp U}, $\cU^F_{E\vee F} \ast \cB^{E\vee F}_{\bm{0}} \subset \cUc^{F\perp}_{\bm{0}}$ and so the pairing vanishes as desired.


\section{Connection to Brylawski-Schechtman-Varchenko determinant formula}
\label{sec:BSV form}

In~\cite{SV}, Schechtman-Varchenko defined a bilinear form, analogous to the classical Shapovalov form, on a certain flag space associated to a hyperplane arrangement, and computed its determinant.  This bilinear form and determinant formula are further generalized to the setting of arbitrary matroids in work of Brylawski-Varchenko~\cite{BV}.  In this section, our aim is to compare the bilinear form of Brylawski-Schechtman-Varchenko to the restriction of the bilinear form we have defined here on $\cB$ to the space $\cU$.  Our interest in the form on $\cU$ comes from the fact, shown in Section \ref{sec:characters}, that its rank determines the characters of the simple $R$-modules, and in particular, whether $R$ is semisimple.

Since our bilinear form is defined on one graded piece at a time, we can assume that $M$ is coloop-free and consider just the largest piece $\cB(M)$.

The Brylawski-Schechtman-Varchenko (BSV) form is defined via the Orlik-Solomon algebra.  The $r$th graded piece of the Orlik-Solomon algebra $A^r$, is defined as the quotient of $\cB(M)$ by the image of $\partial_v: \Lambda^{r+1}_r(M) \to \cB(M)$.  Thus $A^r = \cB(M)/\Im \partial_v$.
The BSV form is defined on the \emph{flag space} $\cF^\bullet(M)$ of the matroid.  
We will be interested in the piece $\cF^r(M)$ of top degree; by \cite[(2.7)]{BV} 
this space is naturally identified with the dual space $A^r(M)^*$.  Using the formula
\cite[(3.5)]{BV} it follows that the pairing on $\cF^d$ is the pullback of the pairing on 
$\cB$ by the map
\begin{equation}\label{Flag space}
\cF^r(M) \to A^r(M)^* \hookrightarrow \cB(M)^* \cong \cB(M),
\end{equation}
where the last identification is via the pairing.

\begin{theorem} There are natural isomorphisms $\cUc(M) \cong \cF^r(M)$ and $\cU(M) \cong \cF^{n-r}(M^*)$, which intertwine the restriction of the pairing $\la,\ra$ from $\cB(M)$ to $\cUc(M)$ and $\cU(M)$ and the BSV form on the flag spaces $\cF^{r}(M)$ and $\cF^{n-r}(M^*)$.
\end{theorem}

\begin{proof}  
By definition, the flag space $\cF^r(M)$ endowed with the BSV form is isomorphic to $A^r(M)^* \subset \cB(M)^*$ with the restriction of the form on $\cB(M)^* \cong \cB(M)$.  In other words, we have bilinear form preserving isomorphisms $\cF^r(M) \cong (\cB/\Im \partial_v)^* = (\Im \partial_v)^\perp \subset \cB(M)^* \cong \cB(M)$. By the dual statement of Proposition \ref{prop-Uperp}, there is an equality $(\Im \partial_v)^\perp = \cUc(M) \subset \cB(M)$.  We conclude that there is a natural isomorphism $\cUc(M) \cong \cF^r(M)$ which intertwines the restriction of the bilinear form on $\cB(M)$ to $\cUc(M)$ with the BSV form on $\cF^r(M)$.  The other statement follows by duality.
\end{proof}

We now consider the determinant of the restriction of the bilinear form to $\cU(M)$.  It is defined as the determinant of the matrix of the bilinear form with respect to any fixed basis of the finitely generated torsion-free $R$-module $\cU(M)$.

\begin{corollary} \label{cor-det}
The determinant of the bilinear form on $\cU(M)$ is given by
\[D = \prod_{K \in \F \setminus \{I\}} \left(\sum_{i \in I \setminus K} a(i) \right)^{\beta(M/K) \mu^+(K)}.\]  
\end{corollary}

\begin{proof}
The main result of \cite[Theorem 4.16]{BV} states that the determinant of the BSV form on $\cF^{r}(M)$ is given by
\[\prod_{K \in \F \setminus \{\emptyset\}} \left(\sum_{i \in K} a(i) \right)^{\beta(K) \mu^+(M/K)}.\]
By the previous theorem, the determinant of the form on $\cU$ is equal to the BSV form for the dual matroid.
\end{proof}

\section{Computing characters}
\label{sec:characters}

In this section we assume that $k$ is a field.  To give a sense of the behavior of the quasi-hereditary algebras defined above, in this section we use standard theory to compute the characters of their simple modules and (equivalently) the composition series multiplicities of standard modules in some simple examples.

Consider an algebra $R$ defined via the general setup of Section \ref{sec:Ringel data}.  For any $R$-module $N$ and $F \in \F$, consider $\pi_F N$, which we view as an analogue of a weight space.

Let $\Z[\F]$ be the free abelian group with generators $(e(F))_{F \in \F}$.  If $k$ is a field, for any $R$-module $N$, we define its character $\ch(N) \in \Z[\F]$ as
\[ \ch(N) = \sum_{F \in \F} \dim(\pi_F N) e(F). \]
For example, $\ch(\Delta_E) = \sum_F \dim \cU_E^F e(F)$, using Corollary \ref{cor:standard module}.

For any $F \in \F$, consider the restriction of the bilinear form on $\cB_F$ to the standard submodule (or cell module) $\cU_F \subset \cB_F$.  Let $L(F) = \cU_F/\mathrm{rad}(F)$, where $\mathrm{rad}(F)$ denotes the radical of the bilinear form on $\cU_F$.  By the general theory of cellular algebras~\cite[Section 3]{GrLe}, $L(F)$ is a simple $R$-module and the set $\{L(F)|F \in \F\}$ is a complete set of simple $R$-modules.

\begin{lemma}
\label{lem:char}
For any $F \in \F$, The character of the simple $R$-module $L(F)$ can be expressed as:
\[ \ch(L(F)) = \sum_{E \in \F, E \leq F} \rk\left(\langle~,~\rangle|_{\cU_F^E} \right)e(E).\]
\end{lemma}

\begin{proof}
As $L(F) = \cU_F/\mathrm{rad}(F)$, it suffices to recall that the weight space decomposition of $\cU_F = \oplus_E \pi_E~\cU_F = \oplus_E \cU^E_F$ is orthogonal with respect to the inner product.
\end{proof}

On the other hand, via the equations:
\[ \ch(\Delta(F)) = \sum_{E \in \F, E\leq F} [\Delta(F):L(E)] \ch(L(E)),\]
we see that knowing the characters $\{\ch(L(E))|E\in \F\}$ is equivalent to knowing the composition series multiplicities $[\Delta(F):L(E)]$.

\subsection{Semisimple characteristics}

\begin{theorem}
\label{cor-semisimple}
The algebra $R$ associated to $M$ and $a: I \to k^\times$ is semi-simple if and only if $p$ does not divide $\sum_{i \in F \setminus K} a(i)$ for any $F < K \in \F$ for which $M(F)/K$ is connected. 
\end{theorem}

\begin{proof}
Note $R$ is semisimple if and only if $\Delta(F) = L(F)$ for any $F \in \F$.  By Lemma~\ref{lem:char}, this is the case if the bilinear form on $\cU_F^E$ has full rank for any $E,F \in \F$, which is the case if and only of the determinant of the form is non-zero.  By Corollary~\ref{cor-det}, the determinant of the form on $\cU_F^E$ is \[\prod_{F< K \leq E} \left(\sum_{i \in F \setminus K} a(i) \right)^{\beta(M(F)/K) \mu^+(K)}.\]
As $\mu^+(M)$ is a positive integer for any $M$ and $\beta(M)=0$ if and only if $M$ is connected, all the determinants are nonzero if and only if $p$ does not divide $\sum_{i \in F \setminus K} a(i)$ for any $F<K$ such that $M(F)/K$ is connected.
\end{proof}

\subsection{Examples}

Using the above method, one can theoretically compute characters for matroidal Schur algebras.  We conclude this section with some small examples.

\subsubsection{The multiple point matroid $M_n$}

Let $M_n$ denote the matroid on $n$ elements whose bases are the single element subsets.  The only coloop-free flats of $M_n$ are $\bm{1}$ and $\bm{0}$.  By Corollary~\ref{cor-semisimple}, $R(M_n)$ is semisimple if and only if $p$ does not divide $n$.

Recall that the dimension of the space $\dim \cU^{\bm{1}}_{\bm{0}} = \mu^+(M_n^*) = n-1$.  Thus $\ch(\Delta(\bm{1})) = e(\bm{1})+ (n-1) e(\bm{0})$.

Characters: 
\[
\ch(L(\bm{1})) = 
\begin{cases} 
e_{\bm{1}} + (n-2)e_{\bm{0}} & \mathrm{if}\; p \mid n  \\
e_{\bm{1}} + (n-1)e_{\bm{0}}& \mathrm{otherwise}  \end{cases}
\]

Multiplicity:
\[ [\Delta(\bm{1}):L(\bm{0})] = \begin{cases} 
1 & \mathrm{if}\; p \mid n  \\
0 & \mathrm{otherwise}  \end{cases}
\]

\subsubsection{The dual $M_n^*$ of the multiple point matroid}

Again, the only coloop-free flats of $M_n$ are $\bm{1}$ and $\bm{0}$, and $R(M_n^*)$ is semisimple if and only if $p$ does not divide $n$.

We now have $\dim \cU^{\bm{1}}_{\bm{0}} = \mu^+(M_n) = 1$.  Thus $\ch(\Delta(\bm{1})) = e(\bm{1})+ e(\bm{0})$.

Characters: 
\[
\ch(L(\bm{1})) = 
\begin{cases} 
e_{\bm{1}} + e_{\bm{0}} & \mathrm{if}\; p \mid n  \\
e_{\bm{1}} & \mathrm{otherwise}  \end{cases}
\]

Multiplicity:
\[ [\Delta(\bm{1}):L(\bm{0})] = \begin{cases} 
1 & \mathrm{if}\; p \mid n  \\
0 & \mathrm{otherwise}  \end{cases}
\]

\subsubsection{The graphical matroid $M_G$ for the complete graph $G=K_4$}

In this case, in addition to $\emptyset$ and $I$, there are four intermediate coloop-free flats, each of rank 2 and isomorphic to $M_3^*$.  We denote them by $A_i$ for $i=1,2,3,4$.  Note that $M_G/A_i = M_3$ for any $i$.

We have $\dim \cU^{\bm{1}}_{\bm{0}} = \mu^+(M_G^*) = 6$.  Thus 
\[\ch(\Delta(\bm{1})) = e(\bm{1})+ e(A_1)+e(A_2)+e(A_3)+e(A_4)+ 6 e(\bm{0}).\]
\[\ch(\Delta(A_i)) = e(A_i)+ 2e(\bm{0}).\]

Characters:
\[
\ch(L(\bm{1})) = 
\begin{cases} 
e(\bm{1})+ e(A_1)+e(A_2)+e(A_3)+e(A_4)+ 4 e(\bm{0}) & \mathrm{if}\; p=2  \\
e(\bm{1})+ 3 e(\bm{0}) & \mathrm{if}\; p=3  \\
e(\bm{1})+ e(A_1)+e(A_2)+e(A_3)+e(A_4)+ 6 e(\bm{0}) & \mathrm{otherwise}  \end{cases}
\]

Multiplicity:
\[ [\Delta(\bm{1}):L(\bm{0})] = \begin{cases} 
2 & \mathrm{if}\; p = 2  \\
3 & \mathrm{if}\; p = 3  \\
0 & \mathrm{otherwise}  \end{cases}
\]

\subsection{A Jantzen-Type Sum Formula}
Jantzen's sum formula is a useful tool for computing the characters of simple modules for reductive groups in small cases.  Here we observe that there is a Jantzen-type sum formula in our setting as well. As the proof is basically identical to that of Jantzen's case, we simply state the formula and omit the proof.  For Jantzen's formula and proof, see \cite[Prop. II.8.19]{Jan}.

Let $M$ be a matroid, $k$ a finite localization of $\Z$ and $a: I \to k^\times$ a weight function.  Let $R$ be the corresponding quasi-hereditary $k$-algebra defined by the construction in Section \ref{sec:matroid}.

Let $\nu_p: k \to \Z$ denote the $p$-adic valuation. 

\begin{theorem}
Let $\mathbb F$ be a field with $\mathrm{char}(\mathbb{F}) = p > 0$ such that $p$ is a place of $A$.  For each $F \in \F$, there is a filtration of $R$-modules
\[ \Delta(E) = \Delta(E)^0 \supset \Delta(E)^1 \supset \Delta(E)^2 \supset \dots\]
such that
\[ \sum_{i>0} \ch \Delta(E)^i = \sum_{K >E} \beta(M(E)/K) \nu_p \left(\sum_{i\in E \setminus K} a(i)\right) \ch \Delta(K),\]
and
\[ \Delta(E)/\Delta(E)^1 \cong L(E).\]
\end{theorem}


\section{Kook-Reiner-Stanton convolution formula}
\label{sec:KRS formula}

In~\cite[Eq. (2.1)]{KRSlap}, Kook-Reiner-Stanton show that
\[ |\mathrm{bases~of~}M| = \sum_{\mathrm{flats~of~}M} \left| \substack{\mathrm{internally~passive}\\ \mathrm{bases~}B_1~\mathrm{of}~V }\right| \times \left| \substack{\mathrm{externally~passive} \\ \mathrm{bases~}B_2~\mathrm{of}~M/V }\right|. \]
In fact, they realize this equation as an equality of dimensions coming from a spectral decomposition for a Laplacian operator on the vector space $\cB(M)$. 

The formula~\cite[Eq. (2.1)]{KRSlap} can also be seen as appearing in the representation theory of the algebra $R_k(M)$, whenever the algebra $R_k(M)$ is semi-simple (e.g., if $k$ is a field of characteristic zero).  In this case, the standard modules $\Delta(E) \cong \cU(E)$ are simple, and the $R_k(M)$-module decomposes as a direct sum:
\begin{equation}
\label{eq:springer}
 \cB(M) \cong \bigoplus_{E \in \F} \cU(E) \otimes \cUc(M/E).
 \end{equation}
Note that in comparing dimensions, one recovers~\cite[Eq. (2.1)]{KRSlap}.

If $R_k(M)$ is not semisimple, equation~(\ref{eq:springer}) is not quite true.  Instead, the module $\cB(M)$ is tilting and the equation only holds after passing to the Grothendieck group.  On the other hand, one can consider the decomposition into indecomposable direct summands:
\[\cB(M) \cong \bigoplus_{E \in \F} T(E) \otimes  \widecheck{L}(E),\]
where $\widecheck{L}(E)$ denotes the simple module of $\Rc_k(M)$ corresponding to $E\in \F$.

\begin{question}
Are there combinatorial objects whose counts give the dimensions of $T(E)$ and $\widecheck{L}(E)$?
\end{question}

Motivated by~\cite[Eq. (2.1)]{KRSlap}, Kook-Reiner-Stanton proved the following convolution formula for the Tutte polynomial of a matroid that appears in~\cite{KRSconv}:
\[T_M(x, y)= \sum_{A\subset M} T_{M(A)}(0, y) T_{M/A}(x, 0).\]
Setting the variables $x$ and $y$ to $1$ recovers their original equation.  Proudfoot-Webster~\cite{PW} give a geometric interpretation of this formula under the specialization $y=1$.  

\begin{question}
Does there exist a characteristic $p$ version of the convolution formula, or at least its specialization at $y=1$?
\end{question}

\bibliography{refs}

\def\cprime{$'$} \newcommand{\arxiv}[1]{\href{http://arxiv.org/abs/#1}{\tt
  arXiv:\nolinkurl{#1}}}
\providecommand{\bysame}{\leavevmode\hbox to3em{\hrulefill}\thinspace}
\providecommand{\MR}{\relax\ifhmode\unskip\space\fi MR }
\providecommand{\MRhref}[2]{%
  \href{http://www.ams.org/mathscinet-getitem?mr=#1}{#2}
}
\providecommand{\href}[2]{#2}
\begin{thebibliography}{BLPW10}

\bibitem[AM12]{AM}
Pramod Achar and Carl Mautner, \emph{Sheaves on nilpotent cones, {F}ourier
  transform, and a geometric {R}ingel duality}, Preprint
  \href{http://arxiv.org/abs/1207.7044}{arXiv:1207.7044}, 2012.

\bibitem[Bj{\"o}92]{Bj92}
Anders Bj{\"o}rner, \emph{The homology and shellability of matroids and
  geometric lattices}, Matroid applications, Encyclopedia Math. Appl., vol.~40,
  Cambridge Univ. Press, Cambridge, 1992, pp.~226--283.

\bibitem[BLPW]{BLPWgco}
Tom Braden, Anthony Licata, Nicholas Proudfoot, and Ben Webster,
  \emph{Quantizations of conical symplectic resolutions $\text{II}$: category
  $\mathcal{O}$ and symplectic duality}, \arxiv{1407.0964}.

\bibitem[BLPW10]{GDKD}
\bysame, \emph{Gale duality and {K}oszul duality}, Adv. Math. \textbf{225}
  (2010), no.~4, 2002--2049.

\bibitem[BM]{BMHyperRingel}
Tom Braden and Carl Mautner, \emph{Hyperplane arrangements and {R}ingel
  duality}, in preparation.

\bibitem[BV97]{BV}
T.~Brylawski and A.~Varchenko, \emph{The determinant formula for a matroid
  bilinear form}, Adv. Math. \textbf{129} (1997), no.~1, 1--24.

\bibitem[CPS90]{CPSint}
E.~Cline, B.~Parshall, and L.~Scott, \emph{Integral and graded quasi-hereditary
  algebras. {I}}, J. Algebra \textbf{131} (1990), no.~1, 126--160.

\bibitem[Den01]{Denham}
Graham Denham, \emph{The combinatorial {L}aplacian of the {T}utte complex}, J.
  Algebra \textbf{242} (2001), no.~1, 160--175.

\bibitem[GL96]{GrLe}
J.~J. Graham and G.~I. Lehrer, \emph{Cellular algebras}, Invent. Math.
  \textbf{123} (1996), no.~1, 1--34.

\bibitem[Jan03]{Jan}
Jens~Carsten Jantzen, \emph{Representations of algebraic groups}, second ed.,
  Mathematical Surveys and Monographs, vol. 107, American Mathematical Society,
  Providence, RI, 2003.

\bibitem[KRS99]{KRSconv}
W.~Kook, V.~Reiner, and D.~Stanton, \emph{A convolution formula for the {T}utte
  polynomial}, J. Combin. Theory Ser. B \textbf{76} (1999), no.~2, 297--300.
  \MR{1699230 (2000j:05028)}

\bibitem[KRS00]{KRSlap}
\bysame, \emph{Combinatorial {L}aplacians of matroid complexes}, J. Amer. Math.
  Soc. \textbf{13} (2000), no.~1, 129--148. \MR{1697094 (2001e:05028)}

\bibitem[KX98]{KXstruc}
Steffen K{\"o}nig and Changchang Xi, \emph{On the structure of cellular
  algebras}, Algebras and modules, {II} ({G}eiranger, 1996), CMS Conf. Proc.,
  vol.~24, Amer. Math. Soc., Providence, RI, 1998, pp.~365--386.

\bibitem[KX99]{KXwhen}
\bysame, \emph{When is a cellular algebra quasi-hereditary?}, Math. Ann.
  \textbf{315} (1999), no.~2, 281--293.

\bibitem[Mau14]{Mautner}
Carl Mautner, \emph{A geometric {S}chur functor}, Selecta Math. (N.S.)
  \textbf{20} (2014), no.~4, 961--977.

\bibitem[PW07]{PW}
Nicholas Proudfoot and Benjamin Webster, \emph{Intersection cohomology of
  hypertoric varieties}, J. Algebraic Geom. \textbf{16} (2007), no.~1, 39--63.

\bibitem[SV91]{SV}
Vadim~V. Schechtman and Alexander~N. Varchenko, \emph{Arrangements of
  hyperplanes and {L}ie algebra homology}, Invent. Math. \textbf{106} (1991),
  no.~1, 139--194.

\end{thebibliography}
\bibliographystyle{amsalpha}

\end{document}